\documentclass[11pt]{amsart}
\usepackage{amsmath,amssymb,amsfonts,amsthm,amsopn}
\usepackage{graphics}

 \def\Spnr{Sp(d,\R)}

 \newcommand\mc[1]{{\mathcal{#1}}}

\newcommand{\stft}{short-time Fourier transform}

\newcommand{\fif}{if and only if}

\newtheorem{theorem}{Theorem}[section]
\newtheorem{lemma}[theorem]{Lemma}
\newtheorem{corollary}[theorem]{Corollary}
\newtheorem{proposition}[theorem]{Proposition}
\newtheorem{definition}[theorem]{Definition}

\newtheorem{remark}[theorem]{Remark}

\newcommand{\beqa}{\begin{eqnarray*}}
\newcommand{\eeqa}{\end{eqnarray*}}

\newcommand{\field}[1]{\mathbb{#1}}
\newcommand{\bR}{\field{R}}        
\newcommand{\bN}{\field{N}}        
\newcommand{\bZ}{\field{Z}}        
\newcommand{\bC}{\field{C}}        
\def\la{\lambda}

\def\dim{\mathrm{dim}}
\def\cF{\mathcal{F}}              
\def\cS{\mathcal{S}}

\def\cU{\mathcal{U}}
\def\cA{\mathcal{A}}

\def\a{\aleph}

\def\rd{\bR^d}

\def\rdd{{\bR^{2d}}}

\def\lrd{L^2(\rd)}

\def\intrd{\int_{\rd}}

\def\R{\right)}

\def\<{\left<}
\def\>{\right>}

\def\inv{^{-1}}

\def\mv1{M_v^1}

\def\phas{(x,\xi )}
\def\mn{(m,n)}
\def\mn'{(m',n')}

\def\Spnr{Sp(d,\R)}

\def\o{\xi}
\def\a{\alpha}

\def\N{\mathbb{N}}
\def\R{\mathbb{R}}
\def\Ren{\mathbb{R}^d}
\def\Renn{\mathbb{R}^{2d}}

\def\sch{\mathcal{S}}

\def\H{{\mathbb H}}

\def\f{\varphi}

\def\Sn2{S_{2}(L^{2}(\Ren))}
\def\S1{S_{1}(L^{2}(\Ren))}
\def\sig00{\sigma_{0,0}}

\def\la{\langle}
\def\ra{\rangle}

\newcommand{\A}{\mathcal{A}}

\begin{document}
\begin{abstract} This article gives explicit integral formulas for the so-called generalized metaplectic operators, i.e.  Fourier integral operators (FIOs) of Schr\"odinger type, having a symplectic matrix as their canonical transformation. These integrals are over specific linear subspaces of $\rd$, related to the $d\times d$ upper left-hand side submatrix of the underlying $2d\times 2d$ symplectic matrix. The arguments use the integral representations for the classical metaplectic operators obtained by Morsche and Oonincx in a previous paper, algebraic properties of symplectic matrices and time-frequency tools. As an application, we give a specific integral representation for solutions of the Cauchy problem of Schr\"odinger equations with bounded perturbations for every instant time $t\in \R$, even at the (so-called) caustic points.
\end{abstract}

\title[Integral Representations for the Class of Generalized Metaplectic Operators]{Integral Representations for the Class of Generalized Metaplectic Operators}

\author{Elena Cordero, Fabio Nicola  and Luigi Rodino}
\address{Department of Mathematics,
University of Torino, via Carlo Alberto 10, 10123 Torino, Italy}
\address{Dipartimento di Matematica,
Politecnico di Torino, corso Duca degli Abruzzi 24, 10129 Torino,
Italy}
\address{Department of Mathematics,
University of Torino, via Carlo Alberto 10, 10123 Torino, Italy}

\email{elena.cordero@unito.it}
\email{fabio.nicola@polito.it}
\email{luigi.rodino@unito.it}


\title[Integral Representations for the Class of Generalized Metaplectic Operators]{Integral Representations for the Class of Generalized Metaplectic Operators}
\maketitle\section{Introduction}
The objective of this study is to find  integral representations for generalized metaplectic operators. Starting from the original idea of extending the usual metaplectic
representation of the symplectic group using a certain class of Fourier integral
operators in Weinstein~\cite{Wein85}, these operators were introduced in \cite{Wiener12}  as examples of  Wiener algebras of Fourier integral operators of Schr\"odinger type (cf. \cite{fio5,fio1,fio3,B18} and the extensive references therein) having  symplectic matrices as canonical transformations. They appear for instance in quantum mechanics, as propagators for solutions to Cauchy problems for  Schr\"odinger equations with bounded perturbations \cite{EFsurvey,MetapWiener13,wavefrontsetshubin13}.  In the work \cite{Wiener12}  generalized metaplectic operators turns out  to be the composition of classical metaplectic operators with pseudodifferential operators with symbols in suitable classes of modulation spaces. Classical metaplectic operators, which are unitary operators on $L^2(\rd)$,  arise as intertwining operators for the Schr\"{o}dinger representation (see the next section for details). \par
Explicit integral representations for classical metaplectic operators,  extending the results already contained in the literature \cite{folland89,Fred77,GS1,GS2,Norbert1999}, were given by Morsche and Oonincx in \cite{MO2002} and applied to energy localization problems and to fractional Fourier transforms in \cite{MO1999}, see also \cite{ALAA94,Gos11,FN,Namias80} and the references therein. The novelty of \cite{MO2002}, with respect to the classical works  \cite{Gos11,folland89}, is the explicit integral representation of  metaplectic operators, covering all possible cases of symplectic matrices. Indeed, the integral representation of metaplectic operators in \cite{Gos11,folland89} covers only the cases of non-singular upper-left or upper-right component of the parameterizing matrix.  This work can be considered as a completion   of the study \cite{Wiener12}, since  integral representations of generalized metaplectic operators are given for all possible cases of symplectic matrices parameterizing the phase function.\par
To make it easier to compare the results obtained in \cite{MO2002} and in this paper we use the same definition of Schr\"{o}dinger representation and symplectic group given in \cite{MO2002}; these definitions are not the same as in \cite{folland89,Wiener12}: to compare these results with the latter works, a symplectic matrix $\A$ must be replaced with its transpose $\A^T$. \par
The symplectic group $\Spnr$ is the subgroup of $2d\times 2d$ invertible matrices $GL(2d,\bR)$,   defined by
\begin{equation}\label{sympM}
\Spnr=\left\{\A\in GL(2d,\R):\;\A J \A^T=J\right\},
\end{equation}
where $J$ is the orthogonal matrix
$$
J=\begin{pmatrix} 0_d&I_d\\-I_d&0_d\end{pmatrix},
$$
(here $I_d$, $0_d$ are  the $d\times d$ identity matrix and null matrix, respectively).
Observe that if $\A$ satisfies \eqref{sympM}, then also the transpose $\A^T$ and the inverse $\A^{-1}$  fulfill \eqref{sympM} and so are symplectic matrices as well. 
Writing $\A\in \Spnr$ in the following $d\times d$ block decomposition:
\begin{equation}\label{block}
\A=\begin{pmatrix} A&B\\C&D\end{pmatrix},
\end{equation}
  Morsche and Oonincx in \cite[Theorem 1]{MO2002} represented a metaplectic operator by using $r$-dimensional integrals, were $r=\dim\, R(B)\in \bN$, $0\leq r\leq d$, is the range of the $d\times d$ block $B$.
Their result is the starting point for our representation formula for  generalized metaplectic operators. \par For a phase-space point   $z=\phas\in\rdd$ and a function $f$ defined on $\rd$,
we call a   time-frequency shift (or phase-space
shift) the operator  $$\pi(z)f(t)=M_\xi T_x f(t)= e^{2\pi i t\cdot \xi} f(t-x),$$
(that is, the composition of the  modulation operator $M_\xi$ with the translation $T_x$). The definition of a generalized metaplectic operator $T$ is based on its kernel decay with respect to the set of phase-space shifts $\pi(z)g$, $z\in\rdd$, for a given window function $g$ in the Schwartz class $\cS(\rd)$. The decay is measured using the smooth polynomial weight $\la z\ra=(1+|z|^2)^{1/2}$, $z\in\rdd$.
\begin{definition}\label{shsh} Consider $\A \in Sp(d,\bR )$, $g\in\cS(\rd)$ and $s\geq 0$.
  A linear operator $T:\cS(\rd)\to\cS'(\rd)$ is a generalized metaplectic operator (in short,  $T\in FIO(\A,s)$) if its kernel
satisfies the decay condition
\begin{equation}\label{asterisco}
|\langle T \pi(z) g,\pi(w)g\rangle|\leq {C}\langle w-\cA
z\rangle^{-s},\qquad   w,z\in\rdd.
\end{equation}
\end{definition}
The union $\bigcup _{\cA \in Sp(d,\bR )} FIO(\cA ,s)$ is called the
class of generalized metaplectic operators and denoted by
$FIO(Sp,s)$.
Simple examples of generalized metaplectic operators are provided by the classical metaplectic operators $\mu(\A)$, $\A\in Sp(d,\bR )$, where $\mu$ is the metaplectic representation recalled below, which (according to our notation) satisfy $\mu(\A)\in\cap_{s\geq 0} FIO(\A^T,s)$ (cf. \cite[Proposition 5.3]{Wiener12}). More interesting examples are provided by composing  classical metaplectic operators with  pseudodifferential operators. A pseudodifferential operator (in the Weyl form) with a symbol $\sigma$ is formally defined as
\begin{equation}\label{Weylform}
\sigma^w(x,D) f(x)=\int_{\rdd} e^{2\pi i(x-y)\cdot \xi} \sigma\Big(\frac{x+y}{2},\xi\Big) f(y) dy\, d\xi.
\end{equation}
We focus on symbols in sub-classes of the Sj\"ostrand class (or modulation space) $M^{\infty,1}(\rdd)$. This  class is a special case of modulation spaces, introduced and studied by Feichtinger in \cite{F1} and later redefined and used to prove the Wiener property for pseudodifferential operators by Sj\"ostrand in \cite{wiener30,wiener31}. The space $M^{\infty,1}(\rdd)$ consists of all continuous functions $\sigma$ on $\rdd$
 whose  norm, with respect to a fixed window $g\in\cS(\rdd)$, satisfies
\begin{equation}\label{otto}
\|\sigma \|_{M^{\infty,1} } = \int_{\rdd}\sup_{z\in\rdd}|\langle \sigma,\pi(z,\zeta)g\rangle |\, d\zeta<\infty.
\end{equation}
Note that in the space $M^{\infty,1}$   even the differentiability property can be lost.
The  scale of modulation spaces under our consideration are denoted by $M^{\infty}_{1\otimes v_s}(\rdd)$, $s\in\R$. They are Banach spaces  of tempered distributions $\sigma\in \cS'(\rdd)$ such that their norm
\begin{equation}
\|\sigma \|_{M^{\infty}_{1\otimes v_s} }=\sup_{z,\zeta\in\rdd}|\langle \sigma ,\pi(z,\zeta )g\rangle|v_s(\zeta)<\infty,
\end{equation}
 where $v_s(\zeta)=\langle \zeta
\rangle^{s}$ (it can be shown that their definition does not depend on the choice of the window $g\in\cS(\rdd)$).  For $s>2d$, they turn out to be spaces of continuous functions contained in the
Sj\"ostrand class  $M^{\infty,1}(\rdd)$. The regularity of the class $M^{\infty}_{1\otimes v_s}(\rdd)$ increases  with the parameter $s$. In particular, $\bigcap_{s>2d} M^{\infty}_{1\otimes v_s}(\rdd)=S^{0}_{0,0}$, the  H\"ormander's class of smooth functions on
$\rdd$ satisfying, for every $\alpha\in \N^{2d}$, $$
|\partial^{\alpha}_{z}\sigma(z)|\leq C_{\alpha},\quad \, z\in\rdd,
$$
for a suitable  $C_\a>0$. \par
In the works \cite{Wiener12, EFsurvey} is proved the following characterization for generalized metaplectic operators:
\begin{theorem} \label{tc2}
(i) An operator $T$ is in $FIO(\A , s)$  \fif\
 there exist  symbols  $\sigma_1$ and $\sigma _2 \in M^{\infty}_{1\otimes v_s}(\rdd)$ such that
 \begin{equation}
   \label{eq:kh23}
T=\sigma_1^w(x,D)\mu(\A) =\mu(\A)\sigma_2^w(x,D).
 \end{equation}
(ii) Let $
\A\in Sp(d,\R) $ be a symplectic
matrix with block decomposition \eqref{block} and such that $ \det A\not=0$. Define
the phase function $\Phi $ as
\begin{equation}\label{fase}\Phi(x,\xi)=\frac12 CA^{-1}x\cdot x+
 A^{-1} x\cdot \xi-\frac12  A^{-1}B\xi\cdot \xi.
\end{equation}
Then $T \in FIO(\A,s)$ \fif\ $T$ can be written as a  type I
Fourier integral operator (FIO), that is an operator in the form
\begin{equation}
  \label{eq:kh22}
Tf(x)=\int_{\rd} e^{2\pi i\Phi(x,\xi)}\sigma(x,\xi)\hat{f}(\xi)\, d\xi,
\end{equation}
with symbol $\sigma\in M^{\infty}_{1\otimes v_s}(\rdd)$.
\end{theorem}
\noindent Integral formulas of the type \eqref{eq:kh22}  are also called Fresnel's formulas \cite{GS1}.\par
The main objective of this paper is to find an  integral representation of the type \eqref{eq:kh22} also when the block $A$ is singular. The $d$-dimensional integral in \eqref{eq:kh22} will be split up into two integrals: an $r$-dimensional integral on the range $R(A)$ of the block $A$, where $r=\dim\,R(A)$, the dimension of the linear space $R(A)$, and  a $(d-r)$-dimensional integral on the kernel $N(A)$ of the block $A$ (observe that $\dim\,N(A)=d-r$).
Let us denote by $\cF_{R(A)}$ the partial Fourier transform of a function $f\in L^1(\rd)$ with respect to the linear space $R(A)$; that is, for $x=x_1+x_2,$ $\xi=\xi_1+\xi_2\in R(A)\oplus N(A^T)$,
\begin{equation}\label{FRA}
\cF_{R(A)}f(\xi)=\int_{R(A)} e^{- 2\pi i x_1\cdot \xi_1}f(x_1+x_2)\,dx_1\quad \xi_1\in R(A).
\end{equation}
Since the $d\times d$ block $A: R(A^T)\to R(A)$ is an isomorphism, we denote by $A^{inv}: R(A)\to R(A^T)$ the pseudo-inverse of $A$. We first show this preliminary result for symplectic matrices.
\begin{lemma}\label{pseudo-inverse} Consider $\A\in\Spnr$ with the $2\times 2$ block decomposition in \eqref{block}. Then the $d\times d$ block $B$ is an isomorphism from $N(A)$ onto $N(A^T)$.
\end{lemma}
We denote by $B^{inv}: N(A^T)\to N(A)$ the pseudo-inverse of $B$. Our main result reads as follows.
\begin{theorem}[Integral Representations for generalized metaplectic operators] \label{main} With the notation introduced before, an operator $T$ is in the class $FIO(\A^T , v_s)$  \fif\
 $T$ admits the following integral representation: for $x=x_1+x_2\in R(A^T)\oplus N(A)=\rd$,  $ \xi_2\in N(A),\,y\in R(A)$,
\begin{align}
  \label{eq:khee}
Tf(x)= & \int_{R(A)}\int_{N(A)}e^{\pi i (A^{inv} B x_1 \cdot x_1 -B^T Dx_2\cdot x_2-CA^{inv}y\cdot y)+2\pi i (x_1\cdot A^{inv} y+ x_2\cdot \xi_2)}\\
&\quad\cdot\quad\sigma(x,A^{inv}y+\xi_2)\cF_{R(A)}{f}(y+(B^{inv})^T\xi_2)\,d\xi_2\, dy,\notag
\end{align}
where the  symbol $\sigma$ is in the class $ M^{\infty}_{1\otimes v_s}(\rdd)$.\\
\end{theorem}
Observe that, if  $y\in R(A)$, then  $A^{inv}y\in R(A^T)$ and for any $\xi_2\in N(A)$,   we obtain $\xi=A^{inv}y+\xi_2\in R(A^T)\oplus N(A)= \rd$.\par
 When either the block $A$ is the null matrix or $A$ is nonsingular, the previous integral representation reduces to the following cases:
\begin{corollary}\label{mainC} The integral representation \eqref{eq:khee} yields the following special cases:\\
(i) If $\dim\,R(A)=0$ (i.e. $A=0_d$), then the operator $T\in FIO(\A^T , v_s)$  \fif\
\begin{equation}
Tf(x)= \intrd e^{-\pi iB^TD x\cdot x+2\pi i Bx\cdot t} \tilde{\sigma}_1(x, t) \,f(t)\,dt,\label{metap0eP}
\end{equation}
for a suitable symbol $\tilde{\sigma}_1\in M^{\infty}_{1\otimes v_s}(\rdd)$.\\
(ii) If $\dim \,R(A)=d$,  then the operator $T\in FIO(\A^T , v_s)$  \fif\
\begin{equation}\label{RAd}Tf(x)= \intrd e^{2\pi i \Phi_T(x,\xi)} \tilde{\sigma}_2(x,\xi)\hat{f}(\xi)\,d\xi
\end{equation}
for a suitable symbol $\tilde{\sigma}_2\in M^{\infty}_{1\otimes v_s}(\rdd)$ and where the phase function
\begin{equation}\label{fase2}\Phi_T(x,\xi)= \frac12 A^{-1} B x\cdot x+ A^{-T} x \cdot  \xi-\frac12 CA^{-1}\xi\cdot \xi
\end{equation}
 is the generating function of the canonical transformation $\cA^T$ (i.e., the integral representation of $T$ in \eqref{eq:kh22}).
\end{corollary}
Applications to the previous formulae can be found  in quantum mechanics. The solutions to Cauchy problems for Schr\"{o}dinger equations   with bounded perturbations,  provided by  pseudodifferential operators $\sigma^w(x,D)$ having symbols  $\sigma$ in the classes $M^{\infty}_{1\otimes v_s}(\rdd)$, are generalized metaplectic operators applied to the initial datum (cf. \cite{EFsurvey}, see also \cite{MetapWiener13}).  So, formula \eqref{eq:khee} can be applied to find an integral representation of such operators.\par
   As simple example, one can consider the following Cauchy problem for the anisotropic  perturbed harmonic oscillator in dimension $d=2$ (see Section \ref{Sesempio} below). For  $x=(x_1,x_2)\in \R\times\R$, $t\in\R$, we study
\begin{equation}\label{5.31eq}
\begin{cases} i  \partial_t
 u =H u,\\\
u(0,x)=u_0(x),
\end{cases}
\end{equation}
where
\begin{equation}\label{5.31eqH}H= -\frac{1}{4\pi}\partial^2_{x_2} +\pi x_2^2 +  V(x_1,x_2),\end{equation}
with  $V\in M^\infty_{1\otimes v_s}(\R^2)$, $s>4$. The initial datum $u_0$ is in $\cS(\R^2)$ or in a suitable rougher modulation space, cf. Section \ref{Sesempio}. The solution  $u(t,x)=e^{-itH}u_0$,
has the propagator  $e^{-itH}$ which turns out to be a one-parameter family of generalized metaplectic operators $FIO(\A_t, s)$, related to the symplectic matrices
\begin{equation}\label{esempio}
\A_t = \begin{pmatrix}1 & 0 & 0 & 0\\0&\cos t&0&\sin t\\
0&0&1&0\\
0&-\sin t&0&\cos t\end{pmatrix}\quad t\in\R.
\end{equation}
For $t\in\R$, the $2\times 2$ block $A_t$ is given by \begin{equation}\label{esempiosub}
A_t = \begin{pmatrix}1 & 0 \\0&\cos t\end{pmatrix}.
\end{equation}
Observe that $\det A_t=\cos t$ so that $A_t$ is a singular matrix whenever $t= \pi/2+ k \pi$, $k\in\bZ$, the so-called \emph{caustics} of the solution. In this case, using formula \eqref{eq:khee}, we are able to give an integral representation as well.

To compare with other results in the literature, we recall \cite[Sec.6-7, Chapter 7]{GS2}, which  provides an overview of the classical results on caustics in the context of spectral asymptotics. The works \cite{GS3,Z} are relevant recent references on Fourier integral operators and their applications, from the point of view of the semiclassical limit, i.e. the limit with the Planck constant $\hbar$ tending to $0$. The book by Zworski \cite{Z} (Chapters 10 and 11 are the most relevant in the context of the current manuscript) nicely complements the book by Folland \cite{folland89}. It presents a current view of the topic with the orientation towards partial differential equations. The book \cite{GS3} addresses directly many issues studied in the current manuscript, in the framework of semi-classical analysis. They study local representations of differential operators,  even at caustics, and  apply  their representations to global asymptotic solutions of hyperbolic equations. We refer to \cite[Chapters 4,5,8]{GS3} for the most relevant results.

\section{Preliminaries and notation}\label{prelim}
Here and in the sequel, for $x,y\in\R^m$, $x\cdot y$ denotes the inner product in $\R^m$. As recalled above, given a matrix $A$, we call $A^T$  the transpose of $A$ and denote by $R(A)$ and $N(A)$ the range and the kernel of the matrix $A$, respectively.\par
 Given  $\A\in \Spnr$ with the $2\times 2$ block decomposition \eqref{block},   from \eqref{sympM} it follows that the four blocks must satisfy the following properties:
\begin{align}
D^TA-B^TC&=I_d\label{e1}\\
A^TC-C^TA&=0_d\label{e2}\\
D^TB-B^TD&=0_d\label{e3}.
\end{align}
Moreover, since also
$$
\A^{-1}=\begin{pmatrix} D^T&-B^T\\-C^T&A^T\end{pmatrix},
$$
is a symplectic matrix,  relations \eqref{e2} and \eqref{e3} for $\A^{-1}$ give
\begin{align}
\label{e4}
CA^{-1}-A^{-T}C^T&=0_d\\
-AB^T+BA^T&=0_d.\label{e5}
\end{align}\par
The metaplectic representation $\mu$ of (the two-sheeted cover of)
the symplectic group arises as intertwining operator between the
 Schr\"odinger representation $\rho$ of the Heisenberg
group $\H^d$ and the representation that is obtained from it by
composing $\rho$ with the action of $\Spnr$ by automorphisms on
$\H^d$. Namely, the Heisenberg group $\H^d$ is the group obtained by defining on
$\R^{2d+1}$  the product law
$$
(z,t)\cdot(z',t')=(z+z',t+t'+\frac{1}{2}\omega(z,z')), \quad z,z'\in\rdd,\,t,\,t'\in\R,
$$
where $\omega$ is  the symplectic form
\begin{equation*}
\omega(z,z')=z\cdot Jz', \qquad z,z'\in\R^{2d}.
\end{equation*}
The Schr\"odinger representation of the group $\H^d$
on $\lrd$ is then defined by
$$
\rho(p,q,t)f(x)=e^{2\pi it}e^{\pi i  p\cdot q}
e^{2\pi i p\cdot x}f(x+q),\quad x,q,p\in\rd,\,t\in\R.
$$
The symplectic group acts on $\H^d$ via
automorphisms that leave the center
$\{(0,t):t\in\R\}\in\H^d\simeq\R$ of $\H^d$ pointwise fixed:
$$
A\cdot\left(z,t\right) =\left(Az,t\right).
$$
Therefore, for any fixed $\A\in Sp(d,\R)$ there is a representation
$$
\rho_{\A^T}:\H^d\to\cU(L^2(\R^d)), \qquad \left(z,t\right)\mapsto
\rho\left(\A^T\cdot(z,t)\right)
$$
whose restriction to the center is a multiple of the identity. By
the Stone-von Neumann  theorem, $\rho_{\A^T}$ is equivalent to $\rho$.
So, there exists an intertwining unitary operator
$\mu(\A)\in\cU(L^2(\R^d))$ such that
\begin{equation}\label{ATr}\rho_{\A^T}(z,t)=\mu(\A)\circ\rho(z,t)\circ\mu(\A)^{-1}\quad
(z,t)\in \H^d.\end{equation}
By Schur's lemma, $\mu$ is determined up to a
phase factor $e^{is}, s\in\R$. Actually,  the phase
ambiguity is only a sign, so that $\mu$ lifts to a
representation of the (double cover of the) symplectic group. \par

An alternative definition of a metaplectic operator (cf. \cite{folland89,Norbert1999,MO2002}), up to a constant $c$, with $|c|=1$, involves a time-frequency representation, the so-called Wigner distribution $W_f$ of a function $f\in L^2(\rd)$, given by
\begin{equation}
\label{eq3232} W_{f}(x,\xi)=\int  e^{-2\pi i y\cdot \xi}
f\left(x+\frac{y}2\right)\overline{f\left(x-\frac{y}2\right)}
\,dy.
\end{equation}
The crucial property of the Wigner distribution $W$  is that it intertwines $\mu(\A)$ and the affine action on
$\rdd$:
\begin{equation}\label{wignermuA}
W_{\mu(\A)f}=W_{f}\circ \A,\quad \A\in Sp(d,\R).
\end{equation}
Since  $W_g=W_f$ if and only if there exists a constant $c\in\bC$, with $|c|=1$, such that  $g=c f$, it is clear that, up to a constant $c$ with $|c|=1$, a metaplectic operator can be defined by the intertwining relation \eqref{wignermuA}. \par
Morsche and Oonincx in \cite{MO2002} use the relation
\eqref{wignermuA} to obtain an integral representation (up to a constant $c\in\bC$, with $|c|=1$) of every metaplectic operator $\mu(\A)$,   $\A\in  Sp(d,\R)$, extending the preceding   results for special symplectic matrices contained in the pioneering work of Frederix \cite{Fred77}, in Folland's book  \cite{folland89} and in Kaiblinger's  thesis \cite{Norbert1999} (see also \cite{Gos11,book,hormander3} and references therein).\par
To state the integral representation for metaplectic operators contained in \cite{MO2002}, we need to introduce some preliminaries (cf. \cite{BI92,BI99,MO2002}).
For a $d\times d$ matrix $A$ and a linear subspace $L$ of $\rd$ with $\dim\,L=r$, $q_L (A)$ denotes the $r$-dimensional volume of the parallelepiped
$$X=\{x\in\rd\,:\,x=\xi_1 A e_1+\cdots +\xi_r A e_r,\,\,0\leq \xi_i\leq 1,\,\,i=1,\dots, r\}
$$
spanned by the vectors $A e_1,\dots Ae_r$, where $e_1,\dots, e_r$ is any orthonormal basis of $L$. If $\dim\, A(L)=\dim \,L=r$, then the $r$-dimensional volume of $X$ is positive, otherwise this volume is zero. The number $q_L(A)$ can be interpreted as a matrix volume as follows. We collect the vectors $e_1,\dots, e_r$ as columns into the $d\times r$ matrix $E=[e_1,\dots,e_r]$. Assuming $\dim\, A(L)=\dim \,L=r$, the matrix $AE$ has full column rank and
$$q_L(A)=\mathrm{vol}\,AE=\sqrt{\det (E^TA^TAE)}.
$$
If $L=\rd$ and $A$ is nonsingular, then $q_L(A)=|\det A|$. \par
The definition of $q_L(A)$ is extended to the following cases:  we set $q_L(A)=1$ either when $L$ is the null space and $A$ is nonsingular or  $A$ is the null matrix and $\dim\,L>0$.
The number $q_L(A)$ appears in the change-of-variables formulas for more dimensional integral as follows.
\begin{lemma}\label{lemma1}
Under the assumptions above, if $\dim\,A(L)=\dim\,L$  we have
\begin{equation}\label{changevar}
\int_{L} \f (Ax)\,dx=\frac1 {q_L(A)}\int_{A(L)}\f(x)\,dx,
\end{equation}
for every function $\f\in\cS(\rd)$ or, more generally,  any function $\f$ for which the above integrals exist.
\end{lemma}
\begin{corollary}\label{changevaryC}
Under the assumptions of Lemma \ref{lemma1}, for any $y\in A(L)$, we have
\begin{equation}\label{changevary}
\int_{L} \f (Ax+y)\,dx=\frac1 {q_L(A)}\int_{A(L)}\f(x)\,dx.
\end{equation}
\end{corollary}
\begin{proof}
It is an immediate consequence of Lemma \ref{lemma1}, since by assumption $\dim\,A(L)=\dim\,L$ so that  $A$ is a an isomorphism from $L$ onto $A(L)$.
\end{proof}

 We associate to  a symplectic matrix $\A$ with block decomposition \eqref{block} a constant
\begin{equation}\label{cost}
c(\A)= \sqrt{\frac{s(A)}{q_{N(A)}(C)}},
\end{equation}
where $s(A)$ denotes the product of the nonzero singular values of the $d\times d$ block $A$, or equivalently
\begin{equation}\label{svalues}
s(A)=q_{R(A^T)}(A).
\end{equation}
The integral representation of a metaplectic operator proved in \cite[Theorem 1]{MO2002} and applied to the matrix
$$\mathcal{B}=\A J= \begin{pmatrix} -B&A\\-D&C\end{pmatrix}$$ gives the following integral representation.
\begin{theorem}\label{aJ}
Consider $\A\in\Spnr$ with the $2\times 2$ block decomposition in \eqref{block} and set $r=\dim\,R(A)$.
Then, for $f\in\cS(\rd)$, the metaplectic operator $\mu(\A)$, up to a constant $c\in\bC$, with $|c|=1$, can be represented as follows:\\
\noindent $(i)$ If $r>0$ then
\begin{equation}\label{metap1e}
\mu(\A)f(x)=c(\A)  \int_{R(A^T)}e^{-\pi i B^TD x\cdot x-\pi i A^TC t\cdot t+2\pi i A^TD x\cdot t} \hat{f}(At-Bx)\, dt.
\end{equation}
\noindent $(ii)$ If $r=0$ then
\begin{equation}
\mu(\A)f(x)=\sqrt{|\det B|} \intrd e^{-\pi iB^TD x\cdot x+2 \pi i Bx \cdot t} f(t)\,dt.\label{metap0e}
\end{equation}
\end{theorem}

In the sequel  the integral representations of metaplectic operators will be always meant ``up to a constant'' $c\in\bC$, with $|c|=1$.
\begin{corollary}\label{detA} Under the assumptions of Proposition \ref{aJ}, if $R(A)=d$, that is the block $A$ is nonsingular, then
\begin{equation}
  \label{eq:kh22M}
\mu(\A)f(x)=|\det\, A|^{-1/2}\int_{\rd} e^{2\pi i\Phi_T(x,\xi)}\hat{f}(\xi)\, d\xi,
\end{equation}
where
the phase function $\Phi_T$ is defined in \eqref{fase2}. (Observe that $A^{-1}B$ and $CA^{-1}$ are symmetric matrices by \eqref{e5} and \eqref{e4} respectively).
\end{corollary}
\begin{proof}
Since $A$ is nonsingular, $N(A)=0$ and $R(A^T)=\rd$ so that $c(\A)=\sqrt{|\det A|}$. We make the change of variables $At-Bx=\xi$ in the integrals in \eqref{metap1e} so that $dx= |\det A|^{-1} d\xi$. Making straightforward computations and using the following properties:  the matrix $CA^{-1}$ is  symmetric by relation \eqref{e4}  and  $D-CA^{-1}B=A^{-T}$ by \eqref{e1}, the result immediately follows.
\end{proof}
\begin{remark}\label{33} (i) If $\Phi_T(x,\xi)$ is as in \eqref{fase2}, we have
$$\nabla_{x}\Phi_T(x,\xi)= A^{-1}B x+A^{-1}\xi,\quad \nabla_{\xi}\Phi_T(x,\xi)=A^{-T}x-CA^{-1}\eta$$
and using $D^T=B^T A^{-T}C^T+A^{-1}$ (by relation \eqref{e1}) and $A^{-1}B=B^TA^{-T}$ (by relation \eqref{e5}), we obtain
$$ \begin{pmatrix} x\\ \nabla_{x}\Phi_T\end{pmatrix}=\begin{pmatrix} A^T&C^T\\B^T&D^T\end{pmatrix}\begin{pmatrix} \nabla_{\xi}\Phi_T\\ \xi\end{pmatrix}=\A^T \begin{pmatrix} \nabla_{\xi}\Phi_T\\ \xi\end{pmatrix},
$$
that is the function $\Phi_T$ is the generating phase function of the canonical transformation $\A^T$. Indeed, the phase function $\Phi_T$ in \eqref{fase2} coincides with the generating phase function $\Phi$ in \eqref{fase} when $\A$ is replaced by $\A^T$. The fact we obtain $\A^T$ instead of $\A$ depends on our definition of the Schr\"{o}dinger representation, with follows the one in \cite{MO2002}. Hence, under our notations, $\mu(\A) \in FIO(\A^{T}, v_s)$, for every $s\geq 0$.
Observe that, up to a constant, this is also the integral representation of Theorem $(4.51)$ in \cite{folland89}.\\
(ii) If  $0<\dim\, R(A)=r<d$, then the integral representation in \eqref{metap1e} can be interpreted as a degenerate form of a type I generalized metaplectic operator in $FIO(\A^T,v_s)$, with constant symbol $\sigma=|\det A|^{-1/2}$.\\
(iii)  If $\dim \,R(A)=0$, then \begin{equation}\label{range0}\A=\begin{pmatrix} 0_d&B\\B^{-T}&D\end{pmatrix}\end{equation} and the integral representation in \eqref{metap0e} is, up to a constant factor, the one of Theorem $(4.53)$ in \cite{folland89} (with $\A$ replaced by $\A^T$, so that the block $B$ is replaced by $B^{-1}$ in formula $(4.54)$ of \cite{folland89}).
\end{remark}
\par

We recall the integral representation of Theorem \ref{aJ} for elements of $\Spnr$ in special form, which we shall use in the sequel.
For $f\in \cS(\R^d)$, we have
\begin{align}
\mu\left(\begin{pmatrix}  A&0_d\\ 0_d&\;A^{-T}\end{pmatrix} \right)f(x)
&=\sqrt{|\det A|}f(A x)\label{diag}\\
\mu\left(\begin{pmatrix} I_d&0_d\\ C&I_d \end{pmatrix} \right)f(x)
&= e^{-\pi i Cx\cdot x}f(x)\label{lower}\\
\mu\left(J\right)&=\mc{F}^{-1}\label{iot},
\end{align}
where $\mc{F}$ denotes the Fourier transform
$$
\mc{F} f(\xi)=\int_{\R^d}f(x)e^{-2\pi i x\cdot \xi}\;dx,
\qquad f\in L^1(\R^d).
$$

\subsection{Time-frequency methods}
We recall here the time-frequency tools we shall use to prove the integral representation for generalized metaplectic operators.\par
The polarized version of the Wigner distribution in \eqref{eq3232}, is the so called cross-Wigner distribution $W_{f,g}$, given by
\begin{equation}
\label{eq3232c} W_{f,g}(x,\xi)=\int  e^{-2\pi i y\cdot \xi}
f\left(x+\frac{y}2\right)\overline{g\left(x-\frac{y}2\right)}
\,dy,\quad f,g\in L^2(\rd).
\end{equation}
A pseudodifferential operator in the Weyl form \eqref{Weylform} with symbol $\sigma\in\cS'(\rdd)$ can be also defined by
\begin{equation}\label{Weyl}\la \sigma^w(x,D)f,g\ra=\la \sigma, W(g,f)\ra\quad f,g\in\cS(\rd),\end{equation}
where the brackets $\la\cdot,\cdot\ra$ denote the extension to $\sch' \times\sch $ of
the inner product $\la f,g\ra=\int f(t){\overline {g(t)}}dt$ on
$L^2$.
Observe that by the intertwining relation \eqref{wignermuA} and the definition of Weyl operator \eqref{Weyl}, it follows the property
\begin{equation}\label{weylchirp}
\sigma^w(x,D)\mu(\A)=\mu(\A)(\sigma\circ\A^{-1})^w(x,D).
\end{equation}
{\bf Weighted modulation spaces}.  We shall recall the definition of modulation spaces related to the weight functions
\begin{equation} v_s(z)=\la z\ra^s=(1+|z|^2)^{\frac s 2},\quad  s\in\R.
\end{equation}
Observe that for $\cA\in \Spnr$, $|\cA z|$ defines an equivalent
norm on $\rdd$, hence for every $s\in\R$, there exist $C_1,C_2>0$ such
that
\begin{equation}\label{pesieq}
C_1 v_s(z)\leq v_s(\cA z)\leq C_2 v_s(z),\quad \forall z\in\rdd.
\end{equation}
The time-frequency representation which occurs in the definition of modulation spaces is the short-time Fourier Transform (STFT)
of a distribution $f\in\cS'(\rd)$ with respect to a function $g\in\cS(\rd)$ (so-called window), given by
$$V_gf (z) = \langle f, \pi (z)g\rangle, \quad z=(x,\xi)\in\rdd.$$
 The  \stft\ is well-defined whenever  the bracket $\langle \cdot , \cdot \rangle$ makes sense for
dual pairs of function or distribution spaces, in particular for $f\in
\cS ' (\rd )$, $g\in \cS (\rd )$,  or for $f,g\in\lrd$.

\begin{definition}  \label{prva}
Given  $g\in\cS(\rd)$, $s\geq0$,  and $1\leq p,q\leq
\infty$, the {\it
  modulation space} $M^{p,q}_{1\otimes v_s}(\Ren)$ consists of all tempered
distributions $f\in \cS' (\rd) $ such that $V_gf\in L^{p,q}_{1\otimes v_s}(\Renn )$
(weighted mixed-norm spaces). The norm on $M^{p,q}_{1\otimes v_s}(\rd)$ is
\begin{equation}\label{defmod}
\|f\|_{M^{p,q}_{1\otimes v_s}}=\|V_gf\|_{L^{p,q}_{1\otimes v_s}}=\left(\int_{\Ren}
  \left(\int_{\Ren}|V_gf(x,\o)|^p
    dx\right)^{q/p} v_s(\o)^q d\o\right)^{1/q}  \,
\end{equation}
(with obvious modifications for $p=\infty$ or $q=\infty$).
\end{definition}
 When $p=q$, we  write $M^{p}_{1\otimes v_s}(\rd)$ instead of
 $M^{p,p}_{1\otimes v_s}(\rd)$; when $s=0$ (unweighted case) we simply write  $M^{p,q}(\rd)$ instead of
 $M^{p,q}_{1\otimes 1}(\rd)$. The spaces $M^{p,q}_{1\otimes v_s}(\rd)$ are Banach spaces,  and
 every nonzero $g\in M^{1}_{1\otimes v_s}(\rd)$ yields an equivalent norm in
 \eqref{defmod}, so that their definition is  independent of the choice
 of $g\in  M^{1}_{1\otimes v_s}(\rd)$. We shall use modulation spaces as symbol spaces, so the dimension of the spaces will be $\rdd$ instead of $\rd$. Moreover, in our setting $p=q=\infty$ (similar results occur for symbols in the  weighted Sj\"ostrand classes $M^{\infty,1}_{1\otimes v_s}(\rdd)$, $s\geq 0$).\par
 The modulation spaces $M^{\infty}_{1\otimes v_s}(\rd)$ are invariant under linear and, in particular, symplectic transformations. This property is crucial to infer our main result and  is proved in  \cite[Lemma 2.2]{EFsurvey} (see also \cite[Lemma 2.2]{MetapWiener13}) for the case of symplectic transformations. The proof for linear transformations goes exactly in the same way, just by adding $|\det\A|$  in formula \eqref{eq:kh2},  which is a consequence of a change of variables (observe that $|\det\A|=1$ if $\A$ is a symplectic matrix). We denote by $GL(2d,\R)$ the class of $2d\times 2d$ invertible matrices. Then we can state:
 \begin{lemma} \label{lkh1}
   If $\sigma \in M^\infty _{1\otimes v_s}(\rdd)$ and $\A \in GL(2d,\R)$,
then $\sigma \circ \A \in M^\infty _{1\otimes v_s}(\rdd)$ and
\begin{equation}
  \label{eq:kh2}
\|\sigma \circ \A \inv \|_{M^\infty _{1\otimes v_s}} \leq   |\det\A| \, \|(\A ^T)\inv
\|^s \, \|V_{\Phi
  \circ \A  } \Phi \|_{L^1_{v_s} } \|\sigma   \|_{M^\infty _{1\otimes v_s}},
\end{equation}
where $\Phi\in\cS(\rdd)$ is the window used to compute the norms of $\sigma$ and $\sigma \circ \A \inv $.
\end{lemma}
In the sequel it will be useful to pass from the Weyl to the Kohn-Nirenberg form of a pseudodifferential operator. The latter form can be formally defined by
$$\sigma(x,D) f(x)=\intrd e^{2 \pi i x\cdot \xi} \sigma(x,\xi) \hat f(\xi)\,d\xi,
$$
for a suitable symbol $\sigma$ on $\rdd$.
The previous correspondences are related by $\sigma^w(x,D)=(\mathcal{U}\sigma)(x,D)$, where \begin{equation}\label{U}\widehat{\mathcal{U}\sigma}(\eta_1,\eta_2)=e^{\pi i
  \eta_1\cdot \eta_2}\widehat{\sigma}(\eta_1,\eta_2)
  \end{equation}
   (see, e.g., \cite[formula (14.17)]{book}).
 The classes $M^\infty _{1\otimes v_s}(\rdd)$ are invariant under the action of the unitary operator $\mathcal{U}$, as shown below.
 \begin{lemma}\label{invM} If  $\sigma \in M^\infty _{1\otimes v_s}(\rdd)$ then $\mathcal{U}\sigma \in M^\infty _{1\otimes v_s}(\rdd)$ with
 $$\|\mathcal{U}\sigma\|_{M^\infty _{1\otimes v_s}}\leq C \|\sigma\|_{ M^\infty _{1\otimes v_s}}.$$
 \end{lemma}
 \begin{proof}  Observe that, up to a constant $c$ with $$\mathcal{U} \sigma(z)=\cF^{-1}e^{\pi i z\cdot C z}\cF\sigma=\mu(J \begin{pmatrix}I_{2d} & 0_{2d} \\
-C & I_{2d}
\end{pmatrix}J^T) \sigma=\mu(\mathcal{D})\sigma$$  where
  $C= \begin{pmatrix}0_{d} & 1/2 \,I_{d} \\
1/2\, I_{d} & 0_{d}
\end{pmatrix}$ and $\mathcal{D}= \begin{pmatrix}I_{2d} &C \\
0_{2d} & I_{2d}
\end{pmatrix}\in Sp(2d,\R)$.  Consider now a window function $\Phi\in\cS(\rdd)$.
A straightforward computation shows
$$V_{\mu(\mathcal{D})\Phi} (\mu(\mathcal{D}))\sigma (z,\zeta)=V_\Phi f (\mathcal{D}^{-T}(z,\zeta))=V_\Phi f (z-C \zeta,\zeta).$$
Since $\mu(\mathcal{D})\Phi\in\cS(\rdd)$ and different window functions yield equivalent norms, we obtain
\begin{align*}\|\mathcal{U}\sigma\|_{M^\infty _{1\otimes v_s}}&\leq C \|V_{\mu(\mathcal{D})\Phi} \mu(\mathcal{D})\sigma\|_{L^\infty _{1\otimes v_s}}=\sup_{z,\zeta\in\rdd}
|V_\Phi f (z-C \zeta,\zeta)|v_s(\zeta)\\&=\|V_\Phi \sigma\|_{L^\infty _{1\otimes v_s}}=\|\sigma\|_{M^\infty _{1\otimes v_s}}
\end{align*}
as desired.
 \end{proof}

\section{Integral representations of generalized metaplectic operators}\label{formula}
The aim of this section is to give  integral representations for generalized metaplectic operators $T\in FIO(\A, v_s)$, extending the integral representations \eqref{eq:kh22} in Theorem \ref{tc2}, valid only in the special case $\det A\not=0$.
To obtain integral representations for generalized metaplectic operators $T\in FIO(\A^T,v_s)$, we use the characterization  of generalized metaplectic operators
of Theorem \ref{tc2} and we write $T=\sigma^w(x,D)\mu(\A)$ where $\sigma^w(x,D)$ is a Weyl operator with symbol $\sigma\in M^{\infty}_{1\otimes v_s}(\rdd)$. Then we study the composition of a pseudodifferential operator in the Weyl form with a metaplectic operator whose integral representation is given by Theorem \ref{aJ}.\par
 Define for a $d\times d$  matrix $A$ the pre-image of a linear subspace $L$ of $\rd$:
\begin{equation}\label{spaceL}
\overleftarrow{A}(L)=\{ x\in\rd\,:\, Ax\in L\}.
\end{equation}
The following property  will be useful to study the previous composition.
\begin{proposition}\label{P1}
Assume $\A\in\Spnr$ admits the block decomposition \eqref{block}. Then
\begin{align}
\overleftarrow{C^T}(R(A^T))&=R(A)\label{M1}\\
\dim\, C(N(A))&=\dim\, N(A) \label{M2}\\
\overleftarrow{B}(R(A))&=R(A^T)\label{M4}\\
B^T(N(A^T))&=N(A). \label{M6}
\end{align}
\end{proposition}
\begin{proof} Since the matrix $\mathcal{B}=\A J\in \Spnr$, its block decomposition satisfies \cite[Property 1]{MO2002} which gives relations \eqref{M1} and \eqref{M2}. Analogously, the matrix
$$\mathcal{B}^{-1}= \begin{pmatrix} C^T & -A^T\\-D^T & -B^T\end{pmatrix}\in\Spnr$$
satisfies \cite[Property 1]{MO2002}, so that the other relations are fulfilled.
\end{proof}\par
We are now in position  to prove Lemma \ref{pseudo-inverse}.
\begin{proof}[\bf {Proof of Lemma \ref{pseudo-inverse}}] Observe that by relation \eqref{M4}, $B: N(A)\to N(A^T)$.  By \eqref{e1}, for every $x\in N(A)$ it follows $-B^t Cx=x$, hence $N(A)\subset R(B^T)=N(B)^\bot$. This gives $N(A)\cap N(B)=\{0\}$, so $B$ is an injective mapping and $\dim\, N(A)\leq \dim\, N(A^T)$. Repeating  the same argument for the symplectic matrix \begin{equation}\label{AT}
\A^T=\begin{pmatrix} A^T&C^T\\B^T&D^T\end{pmatrix},
\end{equation}
we obtain  $\dim \, N(A^T)\leq\dim \, N(A)$, hence $\dim \,N(A)= \dim\,  N(A^T)$, i.e., $B$ is onto and its pseudo-inverse $B^{inv}: N(A^T)\to N(A)$ is well-defined.
\end{proof}\par
 Assume  that the matrix $\A\in\Spnr$ admits the block decomposition \eqref{block} with $\dim\,R(A)>0$.  We first  work on the integral representation of $\mu(\A)$ in \eqref{metap1e}.
\begin{theorem}\label{T1}
Consider  $\A\in\Spnr$ with the $2\times 2$ block decomposition in \eqref{block} and assume $\dim\,R(A)=r>0$.  For  $f\in\cS(\rd)$ and
$$x=x_1+x_2\in R(A^T)\oplus N(A)=\rd$$
we have the following integral representation
\begin{equation}\label{metap1eM}
\mu(\A)f(x)= c_1(\A)\int_{R(A)}e^{\pi i [(A^{inv})^T B x_1\cdot x_1-D^TB x_2\cdot x_2]-CA^{inv} y\cdot y+2\pi i A^{inv} y\cdot x_1} \hat{f}(y-Bx_2)\, dy
\end{equation}
where
\begin{equation}\label{c1cost}
c_1(\A)=\frac1{\sqrt{s(A) q_{N(A)}(C)}}.
\end{equation}
\end{theorem}
\begin{proof} Since $\dim\,R(A)=r>0$, the integral representation of $\mu(\A)$ is given by \eqref{metap1e}.  We set
$$Q(x):=\int_{R(A^T)}e^{-\pi i A^TC t\cdot t+2\pi i A^TD x\cdot t} \hat{f}(At-Bx)\, dt.$$
We write $x=x_1+x_2$, with $x_1\in R(A^T)$ and $x_2\in N(A)$. By relation \eqref{M4} we obtain $Bx_1\in R(A)$. Making the change of
variables $y=At-Bx_1$ and applying Corollary \ref{changevaryC}  the integral $Q(x)$ becomes
\begin{align*}
Q(x)=\frac1{q_{R(A^T)}(A)}&\int_{R(A)}e^{-\pi iC (A^{inv}y+A^{inv} B x_1)\cdot (y+Bx_1)+2\pi i D (x_1+x_2)\cdot(y+Bx_1)}\\
&\quad\quad\quad\cdot \hat{f}(y-Bx_2)\, dy,
\end{align*}
where $q_{R(A^T)}(A)=s(A)$.   By the equality \eqref{e1} we obtain $C A^{inv} B x_1-D x_1=-(A^{inv})^T x_1$ and
relation \eqref{e4} yields $(C A^{inv})^T=C A^{inv}$, so that we can write
\begin{align}\label{muA}
\mu(\A)f(x_1+x_2)=\frac{c(\A)}{s(A)} &\,\,e^{\pi i (A^{inv})^T x_1\cdot  B x_1+\pi i (D x_2\cdot Bx_1-Dx_1\cdot B x_2-Dx_2\cdot Bx_2)}\\
&\quad\cdot\quad\int_{R(A)}e^{-\pi iC A^{inv}y\cdot y+2\pi i (A^{inv})^T x_1\cdot y+ Dx_2\cdot y)}\hat{f}(y-Bx_2)\, dy\notag.
\end{align}
Observe that
$$\frac{c(\A)}{s(A)}=\sqrt{\frac{s(A)}{q_{N(A)}(C)}} \frac{1}{s(A)}=\frac{1}{\sqrt{s(A) q_{N(A)}(C)}},$$
which is \eqref{c1cost}.\par
Now, we shall prove that the $d\times d$ block $D$ satisfies
\begin{equation}
\label{D}
D\,:\, N(A)\to N(A^T).
\end{equation}
First, by Lemma \ref{pseudo-inverse}, $B:\,N(A)\to N(A^T)$ whereas by \eqref{M1} it follows $C^T:  N(A^T)\to N(A)$. Hence, using \eqref{e1}, for $x_2\in N(A)$, we obtain
$$ A^TD x_2=C^T B x_2+ x_2\in N(A).$$
Now $A^T$ maps $R(A)$ onto $R(A^T)$ bijectively, this implies    $Dx_2\in R(A)^\bot=N(A^T)$ and \eqref{D} is proved.
Relation \eqref{D} yields $D x_2 \cdot y=0 $ for $y\in R(A)$ and
$D x_2 \cdot B x_1=0$ since $B x_1\in R(A)$ whenever $x_1\in R(A^T)$. Moreover $C^T B x_1\in R(A^T)$, by relation \eqref{M1}, so that
$$ A^T D=C^TB+I_d :\, R(A^T)\to R(A^T).$$
This gives $D x_1\in R(A)$ whenever $x_1\in R(A^T)$, and
$D x_1 \cdot B x_2=0$, for $B x_2\in N(A^T)=R(A)^\bot$ by relation \eqref{M4}.
These observations allow to simplify the expression of $\mu(\A) f(x_1+x_2)$ in \eqref{muA} and give the representation  \eqref{metap1eM}, as desired.
\end{proof}
\begin{remark}\label{36}  If $\dim\,R(A)=d$, that is $A$ is nonsingular, then $N(A)=\{0\}$,  $R(A)=\rd$, $x_2=0$,  $x=x_1$, $s(A)=|\det A|$, $q_{N(A)}(C)=1$ so that  $c_1(\A)=|\det A|^{-1/2}$. Hence
the integral representation \eqref{metap1eM} coincides with \eqref{eq:kh22M}, as expected.
\end{remark}
We now possess all the instruments to prove our main result.
\begin{proof}[\bf {Proof of Theorem \ref{main}}] By Theorem \ref{tc2}, a linear operator $T$ belongs to the class $FIO(\A^T,v_s)$ if and only if there exists
a symbol $\sigma_1\in M^{\infty}_{1\otimes v_s}$ such that $T=\sigma_1^w(x,D)\mu(\A)$. Consider $\A$ with the block decomposition \eqref{block}.
 Observe that the symbols involved in the sequel are the results of compositions of  symbols in $ M^{\infty}_{1\otimes v_s}(\rdd)$ with  suitable symplectic transformations, so that by Lemma \ref{lkh1} they all belong to the same class $M^{\infty}_{1\otimes v_s}(\rdd)$. \par
 First, assume $0<r=\dim\,R(A)$. We shall prove that the composition $T=\sigma_1^w(x,D)\mu(\A)$ admits the integral representation in \eqref{eq:khee}.
We use the integral representation of the metaplectic operator $\mu(\A)$ in \eqref{metap1eM}. Setting
\begin{equation}\label{Pop}
Pf(x_1+x_2):= \int_{R(A)}e^{-\pi i C A^{inv} y\cdot y+2\pi i A^{inv} y\cdot x_1} \hat{f}(y-Bx_2)\, dy.
\end{equation}
 we will show that \begin{equation}\label{N1}\sigma_1^w(x,D)\mu(\A)f(x_1+x_2)= e^{\pi i [A^{inv}B x_1 x_1\cdot  x_1-B^T Dx_2\cdot x_2]}\sigma_2^w(x,D)Pf(x_1+x_2),\end{equation}
  where, for $x=x_1+x_2,\,\xi=\xi_1+\xi_2\in R(A^T)\oplus N(A)$, we define
\begin{equation}\label{sigma2}\sigma_2(x_1+x_2,\xi_1+\xi_2)=c_1(\A)\sigma_1(x_1+x_2, A^{inv} B x_1+\xi_1-D^T B x_2+\xi_2).\end{equation}
Indeed, define on $ \rd= R(A^T)\oplus N(A)$ the symplectic matrix $\mathcal{C}\in Sp(d,\R)$ as follows
$$\mathcal{C}=\begin{pmatrix}I_r & 0_{d-r} & 0_r & 0_{d-r}\\
0_r & I_{d-r} & 0_r & 0_{d-r}\\
-A^{inv}B & 0_{d-r} & I_r & 0_{d-r}\\
0_r & D^T B & 0_r & I_{d-r}
\end{pmatrix}
$$
(observe that the $d\times d$ block $\begin{pmatrix} -A^{inv}B & 0_{d-r} \\0_r, D^T B \end{pmatrix}$ is a symmetric matrix, by relations \eqref{e5} and \eqref{e3}). The  inverse of $\mathcal{C}$  is
$$\mathcal{C}^{-1}=\begin{pmatrix}I_r & 0_{d-r} & 0_r & 0_{d-r}\\
0_r & I_{d-r} & 0_r & 0_{d-r}\\
A^{inv}B & 0_{d-r} & I_r & 0_{d-r}\\
0_r & -D^T B & 0_r & I_{d-r}
\end{pmatrix}
.$$
We have that $\mu(\mathcal{C}) f(x_1+x_2)= e^{\pi i [A^{inv}B x_1 x_1\cdot  x_1-B^T Dx_2\cdot x_2]} f(x_1+x_2)$, by relation \eqref{lower}, so that
$\sigma_1^w(x,D)\mu(\mathcal{C}) = \mu(\mathcal{C}) (\sigma_1\circ \mathcal{C}^{-1})^w$ by means of \eqref{weylchirp}. The equality
\eqref{N1} immediately follows.\par
Next, we pass from the Weyl to the Kohn-Nirenberg form of a pseudodifferential operator: $\sigma_2^w(x,D)=\sigma_3(x,D)$, for the new symbol  $\sigma_3=\mathcal{U} \sigma_2$ where  $\mathcal{U}$ is defined in \eqref{U}.
 Hence $\sigma_3 \in  M^{\infty}_{1\otimes v_s}(\rdd)$ by Lemma \ref{invM}.
Using $x=x_1+x_2,\,\xi=\xi_1+\xi_2\in R(A^T)\oplus N(A)=\rd$, we can express the operator $\sigma_3(x,D)$ by means of integrals over the subspaces $R(A^T)$ and $N(A)$: for every $\f\in\cS(\rd)$,
  $$\sigma_3(x,D)\f(x_1+x_2)\!=\!\int_{R(A^T)} \int_{N(A)} e^{2\pi i x_2\cdot \xi_2} e^{2\pi i x_1\cdot \xi_1}\sigma_3(x_1+x_2,\xi_1+\xi_2)\hat{\f}(\xi_1+\xi_2)\,d\xi_2\, d\xi_1.$$
 The previous decomposition helps to compute $\sigma_3(x,D)Pf(x)$, where the operator $P$ is defined in \eqref{Pop}. Indeed,
 computing first the integral over $R(A^T)$, we obtain
\begin{align*}\int_{R(A^T)} & e^{2\pi i x_1\cdot \xi_1}\sigma_3(x_1+x_2,\xi_1+\xi_2)\cF_{R(A^T)}(e^{2\pi i A^{inv}y\cdot x_1})(\xi_1)\,d\xi_1\\
&= e^{2\pi i A^{inv}y\cdot x_1}\sigma_3(x_1+x_2,A^{inv}y+\xi_2),\end{align*}
and the expression of $\sigma_3(x,D)Pf(x)$ reduces to
 \begin{align*}\sigma_3(x,D)Pf(x_1+x_2)=\int_{R(A)}& e^{2\pi i x_1\cdot A^{inv} y-\pi i CA^{inv}y\cdot y}\left(\int_{N(A)} e^{2\pi i x_2\cdot \xi_2 }
\sigma_3(x_1+x_2,A^{inv}y+\xi_2)\right.\\& \quad\cdot  \left.\left(\int_{N(A)} e^{-2\pi i \xi_2\cdot t} \hat{f}(y-Bt)\, dt\right)\,d\xi_2\right)\, dy\\
=\int_{R(A)}& e^{2\pi i x_1\cdot A^{inv} y-\pi i CA^{inv}y\cdot y}\left(\int_{N(A)} e^{2\pi i x_2\cdot \xi_2 }
\sigma_3(x_1+x_2,A^{inv}y+\xi_2)\right.\\& \quad\cdot  \left.\left(\frac1{q_{N(A)}(B)}\int_{N(A^T)} e^{2\pi i (B^{inv})^T\xi_2\cdot z} \hat{f}(y+z)\, dz\right)\,d\xi_2\right)\, dy,\end{align*}
where the last equality is the consequence of Lemma \ref{lemma1}, with the change of variables $z=-Bt$ and using  Lemma \ref{pseudo-inverse}. Observe that the transpose of $B^{inv}$, denoted by $(B^{inv})^T$, maps $N(A)$ to $N(A^T)$. Finally, the Fourier inversion formula on the subspace $R(A^T)$ gives the desired result in \eqref{eq:khee}, with symbol  $\sigma=\frac1{q_{N(A)}(B)}\sigma_3$.\par
 Consider now the case $\dim\,R(A)=0$. Then the block $B$ is nonsingular and the matrix $\A$ is the one in \eqref{range0},
whereas the integral representation of $\mu(\A)$ is given by \eqref{metap0e}. Using similar arguments as in the previous case, we
compute $Tf(x)=\sigma_1^w(x,D)\mu(\A)f(x)$. We observe that $\sigma_1^w(x,D)(e^{-\pi i B^T D x\cdot x})= e^{-\pi i B^T D x\cdot x} \sigma_4^w(x,D)$
where $\sigma_4(x,\xi)=\sigma_1(x,\xi-B^TDx)$; this follows by the relation  $\sigma_1^w(x,D)\mu(\mathcal{E})=\mu(\mathcal{E})(\sigma\circ \mathcal{E}^{-1})^w(x,D)$,
with
$\mathcal{E}=\begin{pmatrix}I_d & 0_d\\B^TD &I_d\end{pmatrix}\in\Spnr$ and $\mathcal{E}^{-1}=\begin{pmatrix}I_d & 0_d\\-B^TD &I_d\end{pmatrix}.$
Next we rewrite $\sigma_4^w(x,D)$ in the Kohn-Nirenberg form $\sigma_5(x,D)$, with $\sigma_5=\mathcal{U}\sigma_4$, and the operator $\mathcal{U}$ defined in
\eqref{U}. Finally, since $\sigma_5(x,D) (e^{2\pi i x\cdot B^T t})=\sigma_5(x,B^T t)$, we obtain the representation \eqref{metap0eP}. This formula can be recaptured  from \eqref{eq:khee} when $R(A)=\{0\}$, $y=0$ so that $N(A)=\rd$. The block $B$ is invertible on $\rd$, hence $B^{inv}=B^{-1}$, and making the change of variables $B^{-T}\xi_2=\eta$ we obtain the claim. This completes the proof.
\end{proof}
\begin{proof}[\bf {Proof of Corollary \ref{mainC}}]
Item $(i)$ is already proved in Theorem \ref{main}.\par
$(ii)$ If $\dim \,R(A)=d$, that is the block $A$ is nonsingular, then $N(A)$ is the null space, $A^{inv}=A^{-1}$, the inverse of $A$ on $\rd$, $x_2=\xi_2=0$ so that $x_1=x\in\rd$.  In this case the operator  $T$ reduces to the following representation
\begin{align*}Tf(x)&= \intrd e^{\pi i  A^{-1} B x\cdot x+ 2\pi i x \cdot A^{-1} y-\pi i CA^{-1}y\cdot y}
\sigma(x,A^{-1}y)\hat{f}(y)\,\, dy\\
& =\intrd e^{2\pi i \Phi_T(x,y)} \tilde{\sigma}(x,y)\hat{f}(y)\,dy
\end{align*}
where the phase function $\Phi_T$ is the one defined in \eqref{fase2}. Observe that the phase $\Phi_T$ is the generating function of the canonical transformation $\cA^T$ (see Remark \ref{33}). Moreover, by Lemma \ref{lkh1}, the symbol  $\tilde{\sigma}(x,y)=\sigma(x,A^{-1}y)$ is in $M^{\infty}_{1\otimes v_s}(\rdd)$, whenever $\sigma \in M^{\infty}_{1\otimes v_s}(\rdd)$. We then recapture the integral representation of $T$ in \eqref{eq:kh22}, as expected.
\end{proof}

\section{Applications to Schr\"odinger equations}\label{Sesempio}
We now focus on the Cauchy problem for the anisotropic harmonic oscillator stated in \eqref{5.31eq}.
The main result of \cite{EFsurvey} says that the propagator is a generalized metaplectic operator. Let us first recall this issue. Consider the Cauchy problem
\begin{equation}\label{C1}
\begin{cases} i \displaystyle\frac{\partial
u}{\partial t} =(a^w(x,D)+\sigma^w(x,D))u\\
u(0,x)=u_0(x),
\end{cases}
\end{equation}
where the  hamiltonian  $a^w(x,D)$ is the Weyl quantization of a real-valued homogeneous  quadratic polynomial and $\sigma^w(x,D)$ is a pseudodifferential operator with a symbol $\sigma\in M^{\infty}_{1\otimes v_s}(\rdd)$, $s>2d$. Then, a simplified version of \cite[Theorem5.1]{EFsurvey} reads as follows
\begin{theorem}\label{teofinal}
Consider the Cauchy problem \eqref{C1}  above and set $\tilde{H}=a^w(x,D)+\sigma^w(x,D)$. Then the evolution operator $e^{-it \tilde{H}}$ is a generalized metaplectic operator for every $t\in\bR$.
Specifically, we have
\begin{equation}\label{evolution}
e^{-it{H}} = \mu(\mathcal{A}_t)b^w_{1,t}(x,D)=b^w_{2,t}(x,D)\mu(\mathcal{A}_t),\quad t\in\bR
\end{equation}
for some symbols $b_{1,t},b_{2,t}\in M^{\infty}_{1\otimes v_s}(\rdd)$ and where $\mu(\mathcal{A}_t)=e^{-i t a^w(x,D)}$ is the solution to the unperturbed problem ($\sigma^w(x,D)=0$). In particular, for $1\leq p\leq \infty$, if the initial datum  $ u_0\in M^p$, then $u(t,\cdot )
= e^{-it\tilde{H}}u_0\in  M^p$, for all $t\in\bR$.\end{theorem}
The example \eqref{esempio} falls in this setting. Indeed, consider first the unperturbed problem
\begin{equation}\label{5.31eqU}
\begin{cases} i  \partial_t
 u =H_0 u,\\\
u(0,x)=u_0(x),
\end{cases}
\end{equation}
where $u_0\in \cS(\R^2)$ or in $M^p(\R^2)$, and $H_0= -\frac{1}{4\pi}\partial^2_{x_2} +\pi x_2^2 .$
In this case the propagator is a classical metaplectic operator and the solution is provided by
$$u(t,x_1,x_2)=e^{-i t H_0}u_0(x_1,x_2)=\mu(\A_t)u_0(x_1,x_2),$$
where the simplectic matrices $\A_t$ are defined in \eqref{esempio}. For details, we refer for instance to \cite[Section 4]{EFsurvey} or \cite[Chapter 4]{folland89}. Observe that the $2\times 2$ block in \eqref{esempiosub} is singular when $t= \pi/2+ k \pi$, $k\in\bZ$, (the   so-called caustic points).

We now consider the perturbed problem \eqref{5.31eqH}, where the  potential $V(x_1,x_2)$ is a multiplication operator and so a particular example of a pseudodifferential operator with symbol $\sigma(x_1,x_2,\xi_1,\xi_2)=V(x_1,x_2)\in M^\infty_{1\otimes v_s}(\R^4)$, $s>4$ (observe that $d=2$), which  satisfies the assumptions of Theorem \ref{teofinal}.
Indeed, we choose a window function $\Phi(x,\xi)=g_1(x) g_2(\xi)$, where $g_1,g_2\in\cS(\R^2)$.  The STFT of the symbol then splits as follows:  $$ V_\Phi \sigma(z_1,z_2,\zeta_1,\zeta_2)= V_{g_1} (V) (z_1,\zeta_1)V_{g_2} (1)(z_2,\zeta_2),\quad z_1,z_2,\zeta_1,\zeta_2\in\R^2.$$
 Using
 $\la (\zeta_1,\zeta_2)\ra\leq \la \zeta_1\ra \la \zeta_2\ra$ and the fact that  $1\in S^0_{0,0}\subset M^\infty_{1\otimes v_s}(\R^2)$, for every $s\geq0$, the claim follows.

Hence, the representation of the solution $u(t,x)$ of \eqref{5.31eq} is a generalized metaplectic operator applied to the initial datum $u_0$.  For  $t\not= \pi/2+ k \pi$, $k\in\bZ$, the representation of $u(t,x)$ is provided by the  type I FIO stated in \eqref{eq:kh22}, which in this case reads
$$u(t,x_1,x_2)=\int_{\R^2}e^{2\pi i (x_1\cdot \xi_1+(\sec t) x_2\cdot\xi_2)-\pi i (\tan t) (x_2^2+\xi_2^2)}b_t(x_1,x_2,\xi_1,\xi_2)\hat{u_0}(\xi_1,\xi_2)\,d\xi_1 d\xi_2,
$$
for suitable symbols $b_t\in M^{\infty}_{1\otimes v_{\mu+1}}(\bR^4)$.
We are interested in the caustic points $t= \pi/2+ k \pi$, $k\in\bZ$. The corresponding matrix in \eqref{esempio}  is
$$
\A = \begin{pmatrix}1 & 0 & 0 & 0\\0&0&0&1\\
0&0&1&0\\
0&-1&0&0\end{pmatrix}
\quad \mbox{with \, transpose}\quad
\A^T = \begin{pmatrix}1 & 0 & 0 & 0\\0&0&0&-1\\
0&0&1&0\\
0&1&0&0\end{pmatrix}.
$$
Applying Theorem \ref{main} for the transpose matrix $\A^T$, we observe that
in this case $A=A^T=\begin{pmatrix}1 & 0 \\ 0 & 0\end{pmatrix}.$
The range and the kernel of $A$ are given by $R(A)=R(A^T)=\{(\lambda,0),\, \lambda\in\R\}$ and $N(A)=\{ (0,\nu), \, \nu\in\R\}$. In this case $A^{inv}: R(A)\to R(A^T)$ is the identity mapping.  Observe that $D^T B= \begin{pmatrix}1 & 0 \\ 0 & 0\end{pmatrix}\begin{pmatrix}0 & 0 \\ 0 & -1\end{pmatrix}=\begin{pmatrix}0 & 0 \\ 0 & 0\end{pmatrix}$; if $y\in R(A)$, then $C y =0$ and,  for $x\in R(A)\oplus N(A)$, we have $x=(x_1,x_2)$, $x_1,x_2\in\R$.

Setting $T=u(\pi/2+ k \pi,\cdot)$, the integral representation in \eqref{eq:khee}, for a suitable symbol $b\in M^{\infty}_{1\otimes v_{\mu+1}}(\R^4)$, reduces in this case to
\begin{equation*}Tf(x_1,x_2)=\int_{\R} \int_{\R}e^{2\pi i (x_1 y+ x_2 \xi_2)} b((x_1,x_2),(y,\xi_2))(\cF_1 u_0)(y,-\xi_2) d\xi_2 dy
\end{equation*}
where $\cF_1 u_0(\xi_1,\xi_2)=\int_{\R} e^{-2\pi i \xi_1 t} u_0(t,\xi_2)\,dt$ is the one-dimensional Fourier transform of the initial datum $u_0$ restricted to the first variable $x_1$.

Finally, we observe that, if the symbol $b\equiv 1\in M^{\infty}_{1\otimes v_s}(\R^4)$, for every $s\geq0$, then the operator $T$ reduces to
$$Tu_0(x_1,x_2)=(\cF_2 u_0)(x_1,x_2)$$ the one-dimensional Fourier transform of  $u_0$ restricted to  the second variable $x_2$. This   example of fractional Fourier transform was already studied in \cite[Sec. 6.2]{MO2002}.
%
\section*{Acknowledgment}
We would like to thank Professor K.~Gr{\"o}chenig  for suggesting this field of research, in particular pointing out to our attention the paper \cite{MO2002}. We also thank him for  reading and suggesting improvements to a preliminary version of this work.

The first and the third author have been supported by the Gruppo
Nazionale per l'Analisi Matematica, la Probabilit\`a e le loro
Applicazioni (GNAMPA) of the Istituto Nazionale di Alta Matematica
(INdAM).


\begin{thebibliography}{10}
\bibitem{ALAA94}
T. Alieva, V. López, G. Agullo-López and L. B. Almeida.
\newblock The Fractional Fourier Transform in Optical Propagation Problems. {\it  J. Modern Optics.} 41(5):1037-1044, 1994.

\bibitem{BI92}
A.~Ben-Isra\"el.
\newblock A volume associated with $m\times n$ matrices. {\it Linear Algebra Appl.} 167:87--111, 1992.

\bibitem{BI99}
A.~Ben-Isra\"el.
\newblock The change-of-variables formulas using matrix volume. {\it SIAM J. Matrix Anal. Appl.} 21(1):300--312, 1999.

\bibitem{fio5}
E.~Cordero and F. Nicola.
\newblock Boundedness of Schr\"odinger type propagators on modulation spaces. {\it J. Fourier Anal. Appl.} 16(3):311--339, 2010.

\bibitem{EFsurvey}
E.~Cordero and F. Nicola.
\newblock Schr\"odinger  Equations with  Bounded Perturbations.
\newblock {\em J. Pseudo-Differ. Op. and Appl.}, 5(3):319--341, 2014.

\bibitem{FIOapprox} E.~Cordero, K.~Gr\"ochenig and  F. Nicola.
\newblock Approximation of Fourier Integral Operators by Gabor Multipliers.
\newblock {\em J. Fourier Anal. Appl.}, 18(4):661--684, 2012.

\bibitem{Wiener12} E.~Cordero, K.~Gr\"ochenig, F. Nicola and  L. Rodino.
\newblock Wiener algebras of Fourier integral operators.
\newblock {\em J. Math. Pures Appl.}, 99:219--233, 2013.


\bibitem{MetapWiener13} E.~Cordero, K.~Gr\"ochenig, F. Nicola and  L. Rodino.
\newblock  Generalized Metaplectic Operators and the Schr\"odinger
  Equation  with  a  Potential in the  Sj\"ostrand Class.
\newblock {\em Journal of Mathematical Physics}, 55, 2014. doi: 10.1063/1.4892459.

\bibitem{fio1}
E.~Cordero, F. Nicola and L. Rodino. Time-frequency
analysis of Fourier integral operators. {\it Commun. Pure
Appl. Anal}., 9(1):1--21, 2010.

\bibitem{fio3}
E.~Cordero, F. Nicola and L. Rodino.
\newblock Sparsity of  Gabor representation of Schr\"odinger propagators.
\newblock {\em Appl. Comput. Harmon. Anal.}, 26(3):357--370, 2009.

\bibitem{wavefrontsetshubin13} E.~Cordero,  F. Nicola and  L. Rodino.
\newblock Propagation of the Gabor Wave Front Set for Schr\"odinger Equations with non-smooth potentials.
\newblock {\em Submitted}, 2013.   arXiv:1309.0965

\bibitem{B18} A. C\'ordoba and C. Fefferman. Wave packets and Fourier integral operators. {\em Comm. Partial Differential Equations}, 3(11):979--1005, 1978.

\bibitem{Gos11}
M.~A. de~Gosson.
\newblock {\em Symplectic methods in harmonic analysis and in mathematical
  physics}, volume~7 of {\em Pseudo-Differential Operators. Theory and
  Applications}.
\newblock Birkh\"auser/Springer Basel AG, Basel, 2011.

\bibitem{F1}  H.~G.~Feichtinger,
\newblock Modulation spaces on locally
compact abelian groups,
\newblock {\em Technical Report, University Vienna, 1983,} and also in
\newblock {\em Wavelets and Their Applications},
M. Krishna, R. Radha,  S. Thangavelu, editors,
\newblock Allied Publishers,   99--140, 2003.

\bibitem{FN} H. G. Feichtinger, M. Hazewinkel, N. Kaiblinger, E. Matusiak and M. Neuhauser.
 \newblock Metaplectic operators on $C^n$.
 {\em  Quart. J. Math.} 59(1):15--28, 2008.

\bibitem{folland89}
G.~B. Folland.
\newblock {\em Harmonic analysis in phase space}.
\newblock Princeton Univ. Press, Princeton, NJ, 1989.

\bibitem{Fred77}
G.H.M. Frederix.
\newblock {\em Integral Operators related to symplectic matrices}.
\newblock EUT Report WSK, Dept. of Mathematics, Technological University Eindhoven, 1977.

\bibitem{book}
K.~Gr{\"o}chenig. {\it Foundations of time-frequency
analysis}. Applied and Numerical Harmonic Analysis.
Birkh\"auser Boston, Inc., Boston, MA, 2001.


\bibitem{GS1}
V. Guillemin and S. Sternberg.
{\it Symplectic techniques in physics}. Cambridge University Press, 1984.
\bibitem{GS2} V. Guillemin and S. Sternberg.
{\it Geometric Asymptotics}. Mathematical Surveys and Monographs (Revised Edition), No. 14, American Mathematical Society, 1990.
\bibitem{GS3} V. Guillemin and S. Sternberg.
{\it Semiclassical Analysis}. MIT Online Lecture Notes, 2012.

\bibitem{hormander3}
 L.~H\"{o}rmander.
\newblock {\it The Analysis of Linear Partial Differential
Operators}, Vol. III,
Springer-Verlag, 1985.

\bibitem{Norbert1999}
N.~Kaiblinger.
\newblock Metaplectic representation, eigenfunctions of phase space shifts, and Gelfand-Shilov spaces for lca groups, dissertation. Institut f\"ur Mathematik der Universit\"at Wien, \"Osterreich, 1999.

\bibitem{MO1999}
H.~Morsche and P.J. Oonincx.
\newblock Integral representations of affine transformations in phase space with an application to energy localization problems. CWI report, Amsterdam, 1999.

\bibitem{MO2002}
H.~Morsche and P.J. Oonincx.
\newblock On the Integral Representations for Metaplectic Operators.
\newblock {\it J. Fourier Anal. Appl.}, 8(3):245--257, 2002.

\bibitem{Namias80}
V.~Namias.
\newblock The fractional order Fourier transform and its application to quantum mechanics.
\newblock {\it J. Inst. Math. Appl.}, 25:241--265, 1980.

\bibitem{wiener30}   J. Sj\"ostrand. An algebra of pseudodifferential operators. {\em Math. Res. Lett.}, 1(2):185--192, 1994.

\bibitem{wiener31} J. Sj{\"o}strand.
\newblock Wiener type algebras of pseudodifferential operators.
\newblock In {\em S\'eminaire sur les \'Equations aux D\'eriv\'ees Partielles,
 1994--1995}, pages Exp.\ No.\ IV, 21. \'Ecole Polytech., Palaiseau, 1995.
\bibitem{Wein85}
A.~Weinstein.
\newblock A symbol class for some {S}chr\"odinger equations on {${\bf R}^n$}.
\newblock {\em Amer. J. Math.}, 107(1):1--21, 1985.
\bibitem{Z} M. Zworski.
\newblock {\it Semiclassical Analysis}. Graduate Studies in Mathematics, Vol. 138, American Mathematical Society, 2012.
\end{thebibliography}
\end{document}